\documentclass[11pt]{article}

\usepackage[margin=1in]{geometry}  % set the margins to 1in on all sides
\usepackage{amsmath}               % great math stuff
\usepackage{amsfonts}              % for blackboard bold, etc
\usepackage{amsthm}                % better theorem environments

\newtheorem{theor}{Theorem}[section]
\newtheorem{lmm}[theor]{Lemma}
\newtheorem{prop}[theor]{Proposition}
\newtheorem{cor}[theor]{Corollary}

\newtheorem{defin}{Definition}
\newcommand{\ve}{\varepsilon}

\newcommand{\R}{\mathbb{R}}
\newcommand{\E}{\mathbf{E}}

\newcommand{\pr}{\mathbf{P}}

\providecommand{\keywords}[1]{\textit{Keywords:} #1}
\providecommand{\msc}[1]{\textit{AMS MSC 2010:} #1}

\begin{document}

\nocite{*}

\title{Probability distributions of extremes of self-similar Gaussian random fields}
% with rectangular increments}
%On the probability of exceeding some level by self-similar Gaussian random fields with rectangular increments}

\author{Yuriy  Kozachenko \\
Department of Probability Theory, Statistics and Actuarial
Mathematics \\
Taras Shevchenko National University of Kyiv \\
Volodymyrska 64, Kyiv 01601, Ukraine , E-mail:yvk@univ.kiev.uk
\and Vitalii  Makogin \\
Department of Probability Theory, Statistics and Actuarial
Mathematics \\
Taras Shevchenko National University of Kyiv \\
Volodymyrska 64, Kyiv 01601, Ukraine, E-mail:makoginv@ukr.net}

%\author{V. I. Makogin\\
%Department of Probability Theory, Statistics and Actuarial
%Mathematics \\
%Taras Shevchenko National University of Kyiv \\
%Volodymyrska 64, Kyiv 01601, Ukraine }

\maketitle

\begin{abstract}
We have obtained some upper bounds for the probability distribution of extremes of a self-similar Gaussian random field with stationary rectangular increments that are defined  on the compact spaces.  The probability distributions of extremes for the normalized self-similar Gaussian random fields with stationary rectangular increments defined in $\R^2_+$ have been presented.  In our work we have used the techniques developed for the self-similar fields and based on the classical series analysis of the maximal probability bounding from below for the Gaussian fields.
\end{abstract}

\keywords{distribution of extremes, self-similar random field, finite dimensional distributions, fractional Brownian sheet.}

\msc{60E05, 60E15.}

\section{Introduction}
%Large deviation estimates are by now a standard tool in Asymptotic Convex
%Geometry, contrary to small deviation results.

A self-similar process is a stochastic process that is invariant in
distribution under the suitable scaling of time and space. A random
process $\{X(t),t \in \R\}$ is self-similar with index $H
> 0$ if for all $a > 0$
$\{X(t),t\in \R\}\stackrel{d}{=}\{a^{-H}X(at), t \in \R\},$ where $\stackrel{d}{=}$ denotes equality of the finite-dimensional
distributions. We refer to  Embrechts and Maejima \cite{maej} and Samorodnitsky and Taqqu
\cite{Sam}
for the extensive surveys on results and techniques for self-similar processes.

In this paper we consider the self-similar random fields that are an extension of the self-similar stochastic processes. More precisely, we deal with anisotropic self-similar random fields which means that their indexes of self-similarity are different for different coordinates. We denote $\R_+=[0,+\infty).$
%In this paper we work with an self-similar random fields that are anisotropic.
\begin{defin}
\label{def_taq}
A real valued random field $\{X (\mathbf{t}), \mathbf{t} = (t_1, \dots , t_n) \in \R_+^n\}$ is
self-similar with index $\mathbf{H} = (H_1, \ldots ,H_n) \in (0,+\infty)^n$
if \[\left\{X (a_1t_1, \ldots , a_nt_n), \ \ \mathbf{t} \in \R_+^n\right\} \stackrel{d}{=}\left\{a_1^{H_1}\cdots a_n^{H_n} X (\mathbf{t}), \ \ \mathbf{t} \in \R_+^n \right\},\] for all $a_1 > 0, \ldots , a_n > 0.$
\end{defin}

An interest to the anisotropic self-similar random fields is motivated by the applications
coming from the climatological and environmental sciences (see
\cite{env1,env2}). Several authors have proposed to apply such
random fields for modelling phenomena in spatial statistics,
stochastic hydrology and image processing (see
\cite{env3,env4,env5}).

\begin{defin}
\label{fbs} The normalized fractional Brownian sheet with
Hurst index $\mathbf{H}=(H_1,\ldots,H_n), 0<H_i<1, i=\overline{1,n}$
is the centered Gaussian random field
$B_\mathbf{H}=\{B_\mathbf{H}(\mathbf{t}),\mathbf{t} \in \R_+^n\}$
with a covariance function
\[\E(B_\mathbf{H}(\mathbf{t})B_\mathbf{H} (\mathbf{s})) =
2^{-n}\prod_{i=1}^{n}
\left(|t_i|^{2H_i}+|s_i|^{2H_i}-|t_i-s_i|^{2H_i}\right),\quad
\mathbf{t},\mathbf{s} \in \R_+^n. \]
\end{defin}
This field is  self-similar with index $\mathbf{H}=(H_1,\ldots,H_n)$ by Definition~\ref{def_taq}.

Further in the paper, we  assume that the fields satisfy the
Definition~\ref{def_taq}.
Moreover, we shall consider only the case $n=2$ since switching to the parameter of the higher
dimension is rather technical.

Denote $\mathbf{0}=(0,0).$

\begin{defin}
Let $X=\{X(\mathbf{t}),\mathbf{t}\in \R_+^2\}$ be a self-similar field with index $\mathbf{H}=(H_1,H_2)\in (0,+\infty)^2.$
For any $\mathbf{u}=(u_1,u_2) \in \R_+^2$ and any $\mathbf{v}=(v_1,v_2)\in \R_+^2$ such that $v_1>u_1,$ $v_2>u_2$ define
\[ \Delta_{\mathbf{u}} X(\mathbf{v}) =  X(v_1,v_2)-X(u_1,v_2)-X(v_1,u_2)+X(u_1,u_2).
\]
The field $X$ admits stationary rectangular increments if for
any $\mathbf{u}=(u_1,u_1) \in \R_+^2$
\[\{\Delta_{\mathbf{u}} X(\mathbf{u}+\mathbf{h}),\mathbf{h}\in \R^2_+\}\stackrel{d}{=}\{\Delta_{\mathbf{0}} X(\mathbf{h}),\mathbf{h}\in \R^2_+\}.\]
\end{defin}

The fractional Brownian sheet has the stationary rectangular increments. The proof of this property for
the $\R_+^2$ case can be found in the paper \cite{alp}. A similar
property for the case  $n>2$ can be easily proved as well.

The aim of the paper is to obtain the upper bound for probability distributions of extremes of normalized self-similar Gaussian random fields with the stationary rectangular increments. These probabilities can be used for the estimation of asymptotic growth of sample paths of the fractional Brownian sheet. Furthermore, these probabilities can be applied to investigate the asymptotic behavior of the fractional derivative of the fractional Brownian motion, which is used in the analysis of a non-standard maximum
likelihood estimate for the unknown drift parameter in the stochastic differential equations driven by fractional Brownian motion (see Kozachenko et al. \cite{koz_mish}).

To achieve this goal we use the results from the theory of extremes for the Gaussian processes (Kozachenko et al. \cite{VasKozYamENG}). This theory, in turn, is based on the theory of metric spaces. To apply these results we need to choose the appropriate compact metric space $(\mathbf{T},\rho)$ and to estimate the variance of the increments.  Since we work with anisotropic field we expect that the chosen metric has the different geometric characteristics along different directions. So, we use two metrics $\rho_1(\mathbf{t},\mathbf{s})=\max_{i=1,2} |t_i-s_i|,\mathbf{t},\mathbf{s}\in \mathbf{T}\subset \R^+_2$ %In order to describe the anisotropy, we make use of the following metric on :
and $\rho_2(\mathbf{t},\mathbf{s})=\sum_{i=1,2} |t_i-s_i|^{H_i},\mathbf{t},\mathbf{s}\in \mathbf{T}\subset \R^+_2,$ where $\mathbf{H}=(H_1,H_2)\in(0,1)^2$ is the index of self-similarity of the corresponding random field. The second metric %allows us to use the anisotropic features of the field and it
has played an important role in the studying  the anisotropic Gaussian fields and the self-similar random fields (see \cite{xiao}).

%The aim of this paper is to give a short proof of a large deviation result for supremum of nencentered Gaussian %process over infinite horizon. We study family

%We consider the extreme values of fractional Brownian motions, self-similar Gaussian processes and more general %Gaussian processes which have a trend ?ct? for some constants c,?>0 and a variance t2H. We derive the tail %behaviour of these extremes and show that they occur mainly in the neighbourhood of the unique point t0 where the %related boundary function (u+ct?)/tH is minimal. We consider the case that H<?.

%One of our goal is to obtain the upper bounds for probability distributions of extremes of the normalized self-similar Gaussian random fields with stationary rectangular increments defined on $\R^2_+$. 
The main point in the proofs of this paper is the self-similar property of the fields. This yields the similar behavior of sample paths on compact subsets. From the theory of extremes for the Gaussian processes we get the upper bounds for the probabilities defined in the compact sets. Whence we expand $\R^2_+$ into the union of the compact subsets and apply the inequalities for probabilities   in each subset.
We use the techniques of the self-similar fields  based on the classical series analysis for finding the bound from below of the maximal probabilities for the Gaussian fields. Several results in this paper are obtained by the optimization procedure.

The paper is organized as follows. In Section 2, we present the
probability distributions of extremes of the Gaussian fields defined on the compact spaces and a bound for the variance of its increments in the case of self-similar field. In Section 3 we establish the probability distributions of extremes of the fields defined  on compact metric space $(\mathbf{T},\rho_1)$ and derive the upper bounds for such probabilities of the normalized field defined  on $\R^2_+.$ In Section 4 we obtain the probability distribution of extremes of the normalized self-similar Gaussian field defined on the metric space $(\mathbf{T},\rho_2).$

\section{Probability distributions of extremes of a Gaussian field defined on a compact space}
Let $(\Omega,\mathtt{F}, \pr)$ be a complete probability space satisfying the standard assumptions. It is assumed that
all processes under consideration are defined on this space.

The next theorem follows from Theorem 2.8 of \cite{VasKozYamUKR} or it could
be obtained form Lemma 3.2 of \cite{VasKozYamENG}.
\begin{theor}
\label{main_thrm} Let $(\mathbf{T},\rho)$ be a metric compact space
and $X=\{X(\mathbf{t}),\mathbf{t}\in \mathbf{T})\}$ be a separable centered Gaussian
process. Suppose there exists such a continuous monotonically increasing
function $\sigma: \R_+\rightarrow (0,+\infty), \sigma(0)=0$  that the
following inequality holds
\begin{equation}
\label{cond_sigm} \sup_{\rho(\mathbf{t},\mathbf{s})\le
h}{\left(\E(X(\mathbf{t})-X(\mathbf{s}))^2\right)^{1/2}} \le\sigma(h).
\end{equation}
Let \begin{equation}
\label{defin_beta_gama}
\beta=\sigma\left(\inf_{\mathbf{s}\in \mathbf{T}}\sup_{\mathbf{t}\in
\mathbf{T}}\rho(\mathbf{t},\mathbf{s})\right),~~\gamma=\sup_{\mathbf{u}\in
\mathbf{T}}{\left(\E[X^2(\mathbf{u})]\right)^{1/2}}.
\end{equation}
We denote as $N(\ve)$  %a
%metric massiveness of space $(\mathbf{T},\rho)$, i.e.
the minimal number of closed $\rho$-balls with radius $\ve$ needed to cover the space
$(\mathbf{T},\rho)$. Let $r: [1,+\infty)\rightarrow (0,+\infty)$ be such a continuous function that a function $r(e^t),t>0$ is convex.
If
\begin{equation*}
\label{cond_rn} \int_0^{+\infty}r(N(\sigma^{(-1)}(u)))du<\infty,
\end{equation*}
then for all $\lambda>0,0<p<1,\ve>0$
\begin{equation}
\label{i_e} \begin{gathered}
I_\mathbf{T}(\ve):=\pr\left\{\sup_{\mathbf{t} \in
\mathbf{T}}{|X(\mathbf{t})|>\ve}\right\}\le2
\exp{\left\{\frac12\frac{\lambda^2\gamma^2}{1-p}
+p\frac{\lambda^2\beta^2}{2(1-p)^2}-\lambda\ve\right\}\times}\\
\times r^{(-1)}\left(\frac{1}{\beta p} \int_0^{\beta p}r\left(N(\sigma^{(-1)}(u))\right)du\right).
\end{gathered}
\end{equation}

\end{theor}

We shall minimize the right-hand side of (\ref{i_e}) with respect to $\lambda>0.$
%Inequality (\ref{i_e}) contains two undefined parameters. The statement
%of the next corollary is obtained by minimising of the right hand side of (\ref{i_e}) can with respect to $\lambda>0.$
\begin{cor}
\label{cor_main_th} Under the conditions of Theorem \ref{main_thrm}
we have
\begin{equation}
\label{cor_inequal_1} I_\mathbf{T}(\ve)\le2
\exp{\left\{-\frac{\ve^2(1-p)}{2\left(\gamma^2+\frac{\beta^2p}{1-p}\right)}
\right\}} r^{(-1)}\left(\frac{1}{\beta p} \int_0^{\beta
p}r\left(N(\sigma^{(-1)}(u))\right)du\right).
\end{equation}
\end{cor}
\begin{proof}
Consider the right-hand side of (\ref{i_e}). To prove the corollary it is sufficient to minimize the following value
$$
\frac12\frac{\lambda^2\gamma^2}{1-p}
+p\frac{\lambda^2\beta^2}{2(1-p)^2}-\lambda\ve.$$
Differentiating this expression with respect to $\lambda,$ we get
$$\frac {d}{d\lambda}\left(
\frac12\frac{\lambda^2\gamma^2}{1-p}
+p\frac{\lambda^2\beta^2}{2(1-p)^2}-\lambda\ve\right)=\frac{\lambda\gamma^2}{1-p}
+p\frac{\lambda\beta^2}{(1-p)^2}-\ve.$$
Then, the minimum is achieved if
$$\lambda=\lambda^*=-\frac{\ve(1-p)}{2\left(\gamma^2+\frac{\beta^2p}{1-p}\right)}.$$
If we replace $\lambda$ by $\lambda^*$ in (\ref{i_e}), we obtain (\ref{cor_inequal_1}).
%Substituting $\lambda$ by $\lambda^*$ in (\ref{i_e}) completes the proof of the corollary.
\end{proof}

Throughout the paper the field  $X=\{X(\mathbf{t}),\mathbf{t}\in\R^2_+\}$ is a Gaussian self-similar random field with index $\mathbf{H} =(H_1,H_2)\in (0,1)^2$ and with stationary rectangular increments. Denote $\mathbf{1}=(1,1).$ %Also
%assume that $\E X^2(1,1)=1$, then, obviously,
 Evidently,  $$\E[X(\mathbf{t})]^2=t_1^{2H_1}t_2^{2H_2}\E[X^2(\mathbf{1})],\quad \mathbf{t}=(t_1,t_2)\in \R^2_+.$$
In what follows we need some auxiliary results.
\begin{lmm}
\label{lemma_1}
For all $\mathbf{s}=(s_1,s_2)\in\R^2_+,$ $\mathbf{t}=(t_1,t_2)\in\R^2_+$ we have
\begin{eqnarray}
\label{lemma_ineq_1}
  \E\left[X(\mathbf{t})-X(s_1,t_2)\right]^2 &=& |t_1-s_1|^{2
H_1}t_2^{2H_2}\E
X^2(\mathbf{1}), \\
\label{lemma_ineq_2}
  \E\left[X(s_1,t_2)-X(\mathbf{s})\right]^2 &=& |t_2-s_2|^{2
H_2}s_1^{2H_2}\E X^2(\mathbf{1}).
\end{eqnarray}
\end{lmm}
\begin{proof}
Without loss of generality suppose that $s_1\leq t_1.$ It follows from self-similarity that for any
$s\in \R_+:$ $X(s,0)=X(0,s)=0$ a.s. Then the left-hand side of (\ref{lemma_ineq_1}) equals
$$\E\left(X(\mathbf{t})-X(t_1,0)-X(s_1,t_2)+X(s_1,0)\right)^2=\E\left(\Delta_{s_1,0} X(\mathbf{t})\right)^2.$$
Stationarity of the increments implies that
$$\E\left(\Delta_{s_1,0} X(\mathbf{t})\right)^2=\E\left(\Delta_{\mathbf{0}} X(t_1-s_1,t_2)\right)^2=\E\left(X(t_1-s_1,t_2)\right)^2.$$
Further,  self-similarity implies that
$$\E\left(X(\mathbf{t})-X(s_1,t_2)\right)^2=\E\left(X(t_1-s_1,t_2)\right)^2=|t_1-s_1|^{2
H_1}t_2^{2H_2}\E X^2(\mathbf{1}).$$ The proof of the equality (\ref{lemma_ineq_2}) can be done in a
similar way.
\end{proof}
\begin{lmm}
\label{lemma_2} Assume that $\E X^2(\mathbf{1})=1.$ For all $\mathbf{s}=(s_1,s_2)\in\R^2_+,$ $\mathbf{t}=(t_1,t_2)\in\R^2_+$ we have
\begin{equation}
\label{lemma_ineq3}
\left(\E\left[X(\mathbf{t})-X(\mathbf{s})\right]^2\right)^{1/2} \le
|t_1-s_1|^{H_1}t_2^{H_2}+|t_2-s_2|^{H_2}s_1^{H_1}.\end{equation}
\end{lmm}
\begin{proof}
Using the Minkowski inequality, we get
$$\left(\E\left[X(\mathbf{t})-X(\mathbf{s})\right]^2\right)^{1/2}=
\left(\E\left[X(\mathbf{t})-X(s_1,t_2)+X(s_1,t_2)+X(\mathbf{s})\right]^2\right)^{1/2}$$
$$\le\left(\E\left[X(\mathbf{t})-X(s_1,t_2)\right]^2\right)^{1/2}+\left(\E\left[X(s_1,t_2)-X(\mathbf{s})\right]^2\right)^{1/2}.$$
It follows from Lemma \ref{lemma_1} that
$$\E\left[X(\mathbf{t})-X(s_1,t_2)\right]^2=|t_1-s_1|^{2
H_1}t_2^{2H_2},$$ and
$$\E\left[X(s_1,t_2)-X(\mathbf{s})\right]^2=|t_2-s_2|^{2
H_2}s_1^{2H_2}.$$ Hence, inequality (\ref{lemma_ineq3}) holds.
\end{proof}

\section{Random fields on space $(\mathbf{T},\rho_1)$}
%Probability of exceeding some level on space $(\mathbf{T},\rho_1)$}
In this section we put $\rho(\mathbf{t},\mathbf{s})=\rho_1(\mathbf{t},\mathbf{s})=\max_{i=1,2} |t_i-s_i|,\mathbf{t},\mathbf{s}\in \mathbf{T}\subset \R^+_2.$
%Let us prove corollaries of Theorem \ref{main_thrm} in the case
\begin{cor}
\label{cor_r1_1} Let $\sigma(h)=C h^\alpha, 0<\alpha\le 1, C>0$  and $\mathbf{T}=[0,T]^2$ in Theorem
\ref{main_thrm}. Then
\begin{equation}
\label{cor_inequal_2} I_{[0,T]^2}(\ve)\le8
\exp{\left\{-\frac{\ve^2(1-p)}{2\left(\gamma^2+\frac{C^2T^{2\alpha}p}{2^{2\alpha}(1-p)}\right)}
\right\}}\left(\frac{e}{p}\right)^{2/\alpha}
\end{equation}
for all $0<p<1$ and $\ve>0$.
\end{cor}
\begin{proof}
We have
$$\beta=C\left(\frac{T}{2}\right)^\alpha,~~ N(u)\le
\left(\frac{TC^{1/\alpha}}{2u^{1/\alpha}}+1\right)^2.$$
Put
$r(v)=v^\mu,$ $v\in \R_+,$ $0<\mu<\alpha/2$.  It follows from Corollary \ref{cor_main_th} that
$$I_{[0,T]^2}(\ve)\le2
\exp{\left\{-\frac{\ve^2}{2\left(\gamma^2+\frac{C^2T^{2\alpha}p}{2^{2\alpha}(1-p)}\right)}
\right\}}Z(p),$$ where
\begin{equation}
\label{zp} Z(p)=\left(\frac{1}{\beta p} \int_0^{\beta
p}\left(N(\sigma^{(-1)}(u))\right)^\mu du\right)^{1/\mu}.
\end{equation}
Since $u \le\beta p,$ we have
$$\frac{TC^{1/\alpha}}{2u^{1/\alpha}}\ge\frac{TC^{1/\alpha}}{2(\beta
p)^{1/\alpha}}>\frac{2TC^{1/\alpha}}{2TC^{1/\alpha}}=1.$$ Therefore, we obtain
$$Z(p)
\le\left(\frac{1}{\beta p} \int_0^{\beta
p}\left(\frac{TC^{1/\alpha}}{2u^{1/\alpha}}+1\right)^{2\mu}du\right)^{1/\mu}\le
\left(\frac{1}{\beta p} \int_0^{\beta
p}\left(\frac{TC^{1/\alpha}}{u^{1/\alpha}}\right)^{2\mu}du\right)^{1/\mu}=$$
$$=
T^2C^{2/\alpha}\frac{1}{(\beta
p)^{1/\mu}}\left(\int_0^{\beta
p}\left(\frac{1}{u^{1/\alpha}}\right)^{2\mu}du\right)^{1/\mu}=
T^2C^{2/\alpha}\frac{1}{(\beta
p)^{2/\alpha}}\frac{1}{(1-2\mu/\alpha)^{1/\mu}}.$$ As
$\mu\rightarrow0$, we have
$$Z(p)\le T^2C^{2/\alpha}\frac{1}{(\beta
p)^{2/\alpha}}e^{2/\alpha}=4 \left(\frac{e}{p}\right)^{2/\alpha}.$$
The last inequality completes the proof.
\end{proof}

%In the next proposition we prove that Corollary \ref{cor_r1_1} is true and obtain the estimate on a unit cube $[0,1]^2$.
%consider
From now on we denote $H=\min\{H_1,H_2\},$ where $\mathbf{H}=(H_1,H_2)\in (0,1)^2$ is the index of self-similarity.
\begin{prop}
\label{prop_1} Let $\mathbf{T}=[0,1]^2, \rho=\rho_1,$ and
$X=\{X(\mathbf{t}),\mathbf{t}\in \R_+^2\}$ be a centered Gaussian self-similar
random field of order $\mathbf{H}=(H_1,H_2)\in (0,1)^2$ with stationary rectangular increments. Then for all
$0<p<1$ we have
\begin{equation}
\label{prop_1_ineq}
\pr\left\{\sup_{\mathbf{t}\in[0,1]^2}|X(\mathbf{t})|>\ve\right\}\le
8\exp\left\{-\frac{\ve^2(1-p)}{2\left(1+\frac{4p}{2^{2H}(1-p)}\right)}\right\}\left(\frac
ep\right)^{2/H},\quad\ve>0.
\end{equation}

\end{prop}

\begin{proof}
We have from inequality (\ref{lemma_ineq3}) that for all $\mathbf{t},\mathbf{s}\in [0,1]^2$
$$\left(\E\left[X(\mathbf{t})-X(\mathbf{s})\right]^2\right)^{1/2} \le |t_1-s_1|^{H_1}t_2^{H_2}+|t_2-s_2|^{H_2}s_1^{H_1}$$
$$\leq |t_1-s_1|^{H_1}+|t_2-s_2|^{H_2}\leq 2 \max_{i=1,2} |t_i-s_i|^{H_i}\leq 2 \max_{i=1,2} |t_i-s_i|^{H}=2[\rho(\mathbf{s},\mathbf{t})]^H.$$
Therefore, it follows from (\ref{cond_sigm}) that $\sigma(h)=2h^H$ and $\gamma=1,$ where $\gamma$ is defined in (\ref{defin_beta_gama}).
Thus, inequality (\ref{prop_1_ineq}) follows from (\ref{cor_inequal_2}), where $C=2,T=1,\alpha=H.$
\end{proof}

Denote   $S_{T_1 T_2}=[0,T_1]\times[0,T_2]\subset \R^2_+, \ T_1>0,T_2>0.$
The self-similarity of random field gives a correspondence between the probability distribution of extremes that defined in $[0,1]^2$ and in $S_{T_1 T_2}.$
\begin{cor}
\label{cor_pr_1} Under the conditions of Proposition \ref{prop_1}, we have
\begin{equation}
\label{cor_pr_1_ineq}
\begin{gathered}
\pr\left\{\sup_{\mathbf{t}\in S_{T_1 T_2}}\frac{|X(\mathbf{t})|}{T_1^{H_1}T_2^{H_2}}>\ve\right\}=\pr \left\{\sup_{\mathbf{t}\in[0,1]^2}|X(\mathbf{t})|>\ve\right\}\\
\le 8\exp\left\{-\frac{\ve^2(1-p)}{2\left(1+\frac{4p}{2^{2H}(1-p)}\right)}\right\}\left(\frac
ep\right)^{2/H},
\end{gathered}
\end{equation}
where $\ve>0, p\in (0,1).$
\end{cor}
\begin{proof}
It follows from self-similarity that
$\left\{T_1^{-H_1}T_2^{-H_2}X\left(T_1 t_1,T_2 t_2\right), \\ \mathbf{t}\in \R_+\right\}$ and
$\{X(\mathbf{t}),\\ \mathbf{t}\in \R_+ \}$
have the same finite dimensional distributions. Therefore,
$$\sup_{\mathbf{t}\in S_{T_1 T_2}}\frac{|X(\mathbf{t})|}{T_1^{H_1}T_2^{H_2}}\stackrel{d}{=}\sup_{\mathbf{t}\in[0,1]^2}|X(\mathbf{t})|.$$ Hence, inequality (\ref{cor_pr_1_ineq}) follows from
Proposition \ref{prop_1}.
\end{proof}
%We want to exclude the parameter $p$ from $F(\ve,p)$
\begin{cor}
\label{cor_pr_2} Let $\ve>2.$ Under the conditions of Proposition \ref{prop_1} we have
\begin{equation}
\label{cor_pr_2_ineq}
\pr\left\{\sup_{t\in[0,1]^2}|X(\mathbf{t})|>\ve\right\}\le8e^{\frac
2H+\frac12}\ve^{\frac
4H}\exp\left\{-\frac{3\ve^2}{2\left(4^{1-H}+3\right)}\right\}.
\end{equation}
\end{cor}
\begin{proof}
Put $p=1/{\ve^2}$ in (\ref{prop_1_ineq}). Then
$$\pr\left\{\sup_{t\in[0,1]^2}|X(\mathbf{t})|>\ve\right\}\le 8\exp\left\{-\frac{\ve^2-1}{2\left(1+4^{1-H}(\ve^{2}-1)^{-1}\right)}\right\}
e^{2/H}\ve^{4/H}$$
$$\leq 8\exp\left\{-\frac{3\ve^2}{2\left(3+4^{1-H}\right)}\right\}e^{2/H}\ve^{4/H}e^{\frac{3}{2\left(3+4^{1-H}\right)}}\leq
8e^{\frac
2H+\frac12}\ve^{\frac
4H}\exp\left\{-\frac{3\ve^2}{2\left(4^{1-H}+3\right)}\right\}.$$
The corollary is proved.
\end{proof}
We obtained the upper bound for the probability of exceeding of a self-similar Gaussian random field above the level $\ve>2$ that defined in $[0,1]^2.$

For normalized fields we now prove the upper bound for such probabilities defined in $\R^2_+.$
Denote $x \vee y=\max\{x,y\}.$
\begin{theor}
\label{theorem_3} Let $X=\{X(\mathbf{t}),\mathbf{t}\in \R^2_+\}$ be a centered Gaussian
self-similar random field with index $\mathbf{H}=(H_1,H_2)\in (0,1)^2$ and stationary rectangular increments.
Let a function $c:(0,+\infty)\rightarrow (0,+\infty)$ and a sequence $\{b_n,~n\in \mathbb{N}\bigcup \{0\}\}$ satisfy the following conditions
\begin{itemize}
\item[$(i)$]  $c$ is increasing on $[1,+\infty),$ \quad $c(t)\rightarrow \infty, t\rightarrow\infty,$ \quad and $c\left(\frac{1}{t}\right)=c(t),t\geq1;$
\item[$(ii)$] $b_0=1,$ $b_n<b_{n+1}, n\in \mathbb{N},$ \quad  $b_n\rightarrow\infty, n\rightarrow\infty,$ and $M:=\inf_{k\in0\cup\mathbb{N}}\left(\frac{b_k}{b_{k+1}}\right)^{H_1+H_2}c(b_k)>0;$
\item[$(iii)$] for all $D>0$ the following series converges
\begin{equation*}
\label{theor_series} \sum_{k=1}^{\infty}
\exp\left\{-D\left(\frac{b_k^{H_1+H_2}}{b_{k+1}^{H_1+H_2}}c(b_k)\right)\right\}<+\infty.
\end{equation*}
\end{itemize}
Then for all $\ve>2/M$ we have
\begin{equation}
\label{theor_inequal_3}
\begin{gathered}
\pr\left\{\sup_{t\in\R^2_{+}}\frac{|X(\mathbf{t})|}{(t_1\vee
t_2)^{H_1+H_2}c(t_1\vee t_2)}>\ve\right\} \le \\
\leq 16 e^{\frac{2}{H}+\frac 12}\ve^{4/H}\sum_{k=0}^{\infty}
\exp\left\{-\frac{3\ve^2}{2(4^{1-H}+3)}\left(\frac{b_k^{H_1+H_2}}{b_{k+1}^{H_1+H_2}}c(b_k)\right)^2\right\}
\left(\frac{b_k^{H_1+H_2}}{b_{k+1}^{H_1+H_2}}c(b_k)\right)^{\frac{4}{H}}
=:\widetilde{Z}(\ve).
\end{gathered}
\end{equation}
In this case with probability 1 for all $\mathbf{t}\in (0,+\infty)^2$ the inequality
holds:
$$|X(\mathbf{t})|<\xi(t_1\vee t_2)^{H_1+H_2}c(t_1\vee t_2),$$
where $\xi$ is a random variable such that for all
$\ve>2/M:~\pr\{\xi>\ve\}\le \widetilde{Z}(\ve).$
\end{theor}
\begin{proof}
Denote $$B_k=[0,b_{k+1}]^2 \setminus [0,b_{k})^2,k\geq 0, \quad B_{-k}=\left[0,\frac{1}{b_{k}}\right]^2 \setminus \left[0,\frac{1}{b_{k+1}}\right)^2, \quad k\geq 1.$$
Let us remark that $\mathbf{T}=[0,+\infty)^2=\bigcup_{k=-\infty}^{+\infty}B_k.$ Denote
$$\widetilde{P}(\mathbf{T},\ve)=\pr\left\{\sup_{\mathbf{t}\in \mathbf{T}}\frac{|X(\mathbf{t})|}{(t_1\vee
t_2)^{H_1+H_2}c(t_1\vee t_2)}>\ve\right\}.$$
Evidently, we get
$$\widetilde{P}(\R^2_{+},\ve)\leq \widetilde{P}(\R^2_{+}\setminus [0,1)^2,\ve)+ \widetilde{P}([0,1]^2,\ve) .$$
Firstly, consider $\widetilde{P}(\R^2_{+}\setminus [0,1)^2,\ve).$
Note that, if $\mathbf{t}\in B_k,k\geq 0$ then $b_k\leq t_1\vee t_2\leq b_{k+1}$ and $c(b_k)\leq c(t_1\vee t_2)\leq c(b_{k+1}),k\geq 1.$ Therefore, we get
$$\pr\left\{\sup_{\mathbf{t}\in\R^2_{+}\setminus [0,1)^2}\frac{|X(\mathbf{t})|}{(t_1\vee
t_2)^{H_1+H_2}c(t_1\vee t_2)}>\ve\right\} \leq
\sum_{k=0}^{\infty}\pr\left\{\sup_{\mathbf{t}\in B_{k}}\frac{|X(\mathbf{t})|}{(t_1\vee t_2)^{H_1+H_2}c(t_1\vee
t_2)}>\ve\right\}$$
$$\leq\sum_{k=0}^{\infty}\pr\left\{\sup_{\mathbf{t}\in B_{k}}\frac{|X(\mathbf{t})|}{b_{k+1}^{H_1+H_2}}\frac{b_{k+1}^{H_1+H_2}}{b_k^{H_1+H_2}c(b_k)}>\ve\right\}\leq \sum_{k=0}^{\infty}\pr\left\{\sup_{\mathbf{t}\in [0,b_{k+1}]^2}\frac{|X(\mathbf{t})|}{b_{k+1}^{H_1+H_2}}\frac{b_{k+1}^{H_1+H_2}}{b_k^{H_1+H_2}c(b_k)}>\ve\right\}$$
$$\leq\sum_{k=0}^{\infty}\pr\left\{\sup_{\mathbf{t}\in [0,b_{k+1}]^2}\frac{|X(\mathbf{t})|}{b_{k+1}^{H_1+H_2}}>\ve \frac{b_k^{H_1+H_2}c(b_k)}{b_{k+1}^{H_1+H_2}}\right\}.$$

From corollaries \ref{cor_pr_1} and \ref{cor_pr_2} we obtain  that for $\ve>2/M$
$$\widetilde{P}(\R^2_+\setminus [0,1)^2,\ve) \leq\sum_{k=0}^{\infty}\pr\left\{\sup_{\mathbf{t}\in [0,1]^2}|X(\mathbf{t})|>\ve \frac{b_k^{H_1+H_2}c(b_k)}{b_{k+1}^{H_1+H_2}}\right\}$$
$$\leq 8 e^{\frac{2}{H}+\frac 12}\ve^{4/H}\sum_{k=0}^{\infty}
\exp\left\{-\frac{3\ve^2}{2(4^{1-H}+3)}\left(\frac{b_k^{H_1+H_2}}{b_{k+1}^{H_1+H_2}}c(b_k)\right)^2\right\}
\left(\frac{b_k}{b_{k+1}}\right)^{4\frac{H_1+H_2}{H}}(c(b_k))^{4/H}.
$$

Consider $\widetilde{P}([0,1]^2,\ve).$ Note that, if $\mathbf{t}\in B_{-k},k\geq 1$ then $\frac{1}{b_{k+1}}\leq t_1\vee t_2\leq \frac{1}{b_{k}}$ and $c(b_{k})\leq c(t_1\vee t_2)=c(\frac{1}{t_1\vee t_2})\leq c(b_{k+1}),$ $k\geq 1.$ Therefore, we have
$$\pr\left\{\sup_{\mathbf{t}\in [0,1]^2}\frac{|X(\mathbf{t})|}{(t_1\vee
t_2)^{H_1+H_2}c(t_1\vee t_2)}>\ve\right\}\leq
\sum_{k=1}^{\infty}\pr\left\{\sup_{\mathbf{t}\in B_{-k}}\frac{|X(\mathbf{t})|}{(t_1\vee t_2)^{H_1+H_2} c(t_1\vee t_2)}>\ve\right\}$$
$$\leq \sum_{k=1}^{\infty}\pr\left\{\sup_{\mathbf{t}\in B_{-k}}\frac{|X(\mathbf{t})|}{b_{k+1}^{-H_1-H_2}}\frac{b_{k}^{-H_1-H_2}}{b_k^{-H_1-H_2}c(b_k)}>\ve\right\}\leq \sum_{k=1}^{\infty}\pr\left\{\sup_{\mathbf{t}\in [0,b_{k}^{-1}]^2}\frac{|X(\mathbf{t})|}{b_{k}^{-H_1-H_2}}\frac{b_{k+1}^{H_1+H_2}}{b_k^{H_1+H_2}c(b_k)}>\ve\right\}$$
$$\leq\sum_{k=1}^{\infty}\pr\left\{\sup_{\mathbf{t}\in [0,b^{-1}_{k}]^2}\frac{|X(\mathbf{t})|}{b_{k}^{-H_1-H_2}}>\ve \frac{b_k^{H_1+H_2}c(b_k)}{b_{k+1}^{H_1+H_2}}\right\}.$$

From corollaries \ref{cor_pr_1} and \ref{cor_pr_2} we obtain  that for $\ve>2/M$
$$\widetilde{P}([0,1]^2,\ve) \leq\sum_{k=1}^{\infty}\pr\left\{\sup_{\mathbf{t}\in [0,1]^2}|X(t_1,t_2)|>\ve \frac{b_k^{H_1+H_2}c(b_k)}{b_{k+1}^{H_1+H_2}}\right\}$$
$$\leq 8 e^{\frac{2}{H}+\frac 12}\ve^{4/H}\sum_{k=1}^{\infty}
\exp\left\{-\frac{3\ve^2}{2(4^{1-H}+3)}\left(\frac{b_k^{H_1+H_2}}{b_{k+1}^{H_1+H_2}}c(b_k)\right)^2\right\}
\left(\frac{b_k}{b_{k+1}}\right)^{4\frac{H_1+H_2}{H}}(c(b_k))^{4/H}.
$$
The  theorem is proved.
\end{proof}
The following corollary is an immediate consequence of Theorem \ref{theorem_3}.
\begin{cor}
\label{main_cor_1}
Let $M=\inf_{k\in\{0\}\cup\mathbb{N}}\left(\frac{b_k}{b_{k+1}}\right)^{H_1+H_2}c(b_k)>0.$ Denote
$$u=\frac{3}{(4^{1-H}+3)}\frac{ M^2}{4} \quad \text{ and } \quad v_k=\frac{2}{M^2}\left(\frac{b_k^{H_1+H_2}}{b_{k+1}^{H_1+H_2}}c(b_k)\right)^2,k\geq 0.$$
If for any  $\mathbf{H}\in (0,1)^2$ the series $\sum_{k=0}^{\infty}\frac{v_k^{2/H}}{e^{v_k}}$ converges, then for any $\ve>\frac{2}{M}\sqrt{\frac{2}{3}\left(4^{1-H}+3\right)}$
\begin{equation}
\label{main_cor_1_ineq}
\begin{gathered}
\pr\left\{\sup_{t\in\R^2_{+}}\frac{|X(\mathbf{t})|}{(t_1\vee
t_2)^{H_1+H_2}c(t_1\vee t_2)}>\ve\right\} \leq 16\sqrt{2} \left(\frac{e}{2}\right)^{2/H}
\ve^{4/H}\left(\sum_{k=0}^{\infty}\frac{v_k^{2/H}}{e^{v_k}}\right)M^{4/H} e^{-u \ve^2}.
\end{gathered}
\end{equation}
\end{cor}
\begin{proof}
It is clear that  $u\ve^2>2$ and $v_k>2,k\geq 0.$
Recall that for $u\ve^2,v_k>2$ we have $u\ve^2+v_k\leq u\ve^2 v_k.$
It follows from (\ref{theor_inequal_3}) that for $\ve>\frac{2}{M}\sqrt{\frac{2}{3}\left(4^{1-H}+3\right)}>\frac{2}{s}$ we have
$$
\widetilde{Z}(\ve)= 16 e^{\frac{2}{H}+\frac 12} \ve^{4/H}
\sum_{k=0}^{\infty}\frac{s^{4/H}}{2^{2/H}}\frac{v_k^{2/H}}{\exp\{u\ve^2 v_k\}}\leq 16\sqrt{2} \left(\frac{e}{2}\right)^{2/H}
\ve^{4/H}\left(\sum_{k=0}^{\infty}\frac{v_k^{2/H}}{e^{v_k}}\right)M^{4/H} e^{-u \ve^2}.
$$
The corollary is proved.
\end{proof}
Consider an example of applying Corollary \ref{main_cor_1}.

\textbf{Example 1.}
Put $b_k=e^k,k=0,1,\ldots,$ and $c(t)=\sqrt{\ln\left(|\ln t |+e\right)},t\geq 1$ in Theorem \ref{theorem_3}.
Then $M=\inf_{k\in0\cup\mathbb{N}}\left(\frac{b_k}{b_{k+1}}\right)^{H_1+H_2}c(b_k)=e^{-(H_1+H_2)},$ and
$$u=\frac{3}{(4^{1-H}+3)}\frac{ s^2}{4} =\frac{3}{4(4^{1-H}+3)}e^{-2(H_1+H_2)},$$
$$v_k=2e^{2(H_1+H_2)}\frac{\ln\left(k+e\right)}{e^{2(H_1+H_2)}}=2\ln\left(k+e\right),k\geq 0.$$
Then inequality (\ref{main_cor_1_ineq}) has the form
$$
\begin{gathered}
\pr\left\{\sup_{t\in\R^2_{+}}\frac{|X(\mathbf{t})|}{(t_1\vee
t_2)^{H_1+H_2}\sqrt{\ln(|\ln(t_1\vee t_2)|+e)}}>\ve\right\} \\
\leq 16\sqrt{2} e^{2/H}\ve^{4/H}
\left(\sum_{k=0}^{\infty}\frac{\left(\ln(k+e)\right)^{2/H}}{(k+e)^2}\right) e^{-4\frac{H_1+H_2}{H}} \exp\left\{-u\ve^2\right\}\\\leq 16 \sqrt{2} e^{2/H-8}\ve^{4/H}
\left(\sum_{k=0}^{\infty}\frac{\left(\ln(k+e)\right)^{2/H}}{(k+e)^2}\right) \exp\left\{-\frac{3 \ve^2}{4(4^{1-H}+3)}e^{-2(H_1+H_2)}\right\}.
\end{gathered}
$$
Thus, we obtain the upper bound for probability distribution of extremes of normalized self-similar Gaussian random field with stationary rectangular increments that defined in $\R^2_+$.
\section{Random fields on $(\mathbf{T},\rho_2)$}
Recall the notation of the metric $\rho_2(\mathbf{t},\mathbf{s})=\sum_{i=1,2}|t_i-s_i|^H_i, \mathbf{t}=(t_1,t_2)\in \R^2_+, \mathbf{s}=(s_1,s_2)\in \R^2_+,$ where $H=(H_1,H_2)\in (0,1)$ is the index of self-similarity of a field $X.$
Now we want to obtain result witch is similar to Proposition \ref{prop_1}, but  with metric $\rho_2$.

Let us remember that $N(u)$ is the minimal number of closed $\rho$-balls with radius $u$ needed to cover space
$(\mathbf{T},\rho)$.
First let us prove the estimate for $N(u)$ in the case $\rho=\rho_2$ and $\mathbf{T}=S_{T_1T_2}.$
\begin{lmm}
\label{prop_metric}
Let $\rho=\rho_2$ and $\mathbf{T}=S_{T_1T_2}.$ Then
$$N(u)\leq 2 \left(\frac{T_1}{4 K_1 u^{\frac{1}{H_1}}}+\frac{3}{2}\right)\left(\frac{T_2}{4 K_2 u^{\frac{1}{H_2}}}+\frac{3}{2}\right), u>0,$$
where$$K_1=\left(\frac{H_2}{H_1+H_2}\right)^{\frac{1}{H_1}}, \quad K_2=\left(\frac{H_1}{H_1+H_2}\right)^{\frac{1}{H_2}}.$$
\end{lmm}
\begin{proof}
Consider an auxiliary metric $\rho_3=\{\rho_3(\mathbf{x},\mathbf{y})=\frac{|y_1-x_1|}{a_1}+\frac{|y_2-x_2|}{a_2}, \mathbf{x}=(x_1,x_2)\in \R^2, \mathbf{y}=(y_1,y_2)\in \R^2\},$ with $a_1>0,a_2>0.$
A closed $\rho_3$-ball with radius $1$ in space $(\mathbf{T},\rho_3)$  is a set
$V_{\rho_3}(1)=\left\{\mathbf{x}=(x_1,x_2)\in\R^2, \frac{|x_1|}{a_1}+\frac {|x_2|}{a_2}\leq 1\right\}.$
%$V_{\rho_3}(\ve)=\left\{\mathbf{x}=(x_1,x_2)\in\R^2, \frac{x_1^{H_1}}{a_1}+\frac {x_2^{H_2}}{a_2}\leq\ve\right\}$
The minimum number of $V_{\rho_3}(1)$ needed to cover space
$(\mathbf{T},\rho_3)$ is less then $2\left(\frac{T_1+a_1}{2a_1}+1\right)\left(\frac{T_2+a_2}{2a_2}+1\right).$

Put
$$a_1=2\left(\frac{H_2}{H_1+H_2}\right)^{\frac{1}{H_1}}\ve^{\frac{1}{H_1}}=2K_1 \ve^{\frac{1}{H_1}},$$
$$a_2=2\left(\frac{H_1}{H_1+H_2}\right)^{\frac{1}{H_2}}\ve^{\frac{1}{H_2}}=2K_2 \ve^{\frac{1}{H_2}}.$$
It is not hard to prove that  $V_{\rho_3}(1)\subset V_{\rho_2}(\ve).$
Hence,
$$N_{\rho_2}(\ve)\leq N_{\rho_3}(1)\leq 2 \left(\frac{T_1}{4 K_1 \ve^{\frac{1}{H_1}}}+\frac{3}{2}\right)\left(\frac{T_2}{4 K_2 \ve^{\frac{1}{H_2}}}+\frac{3}{2}\right).$$
\end{proof}
%The next Proposition (?)
To prove the next statement, we need some notation.
Denote $$T_\eta=\max\{T_1^{H_1},T_2^{H_2}\},~H=\min\{H_1,H_2\}, ~ Q=\frac{1}{H_1}+\frac{1}{H_2},$$ $$N_1=\left(\frac{H_1+H_2}{H_2}\right)^{\frac{1}{H_1}}+3, \quad N_2=\left(\frac{H_1+H_2}{H_1}\right)^{\frac{1}{H_2}}+3.$$
\begin{prop}
\label{prop_2} Let
$(\mathbf{T},\rho)=(S_{T_1T_2},\rho_2)$,~$T_1\ge1,T_2\ge1$ and
$X=\{X(\mathbf{t}),\mathbf{t}\in \R^2_+\}$ be a centered self-similar
Gaussian random field with stationary rectangular increments.   Under the conditions of
Theorem \ref{main_thrm}, for all
$0<p<1$ we have
\begin{equation}
\label{prop_2_ineq}
I_{\mathbf{T}}(\ve)=\pr\left\{\sup_{\mathbf{t}\in\mathbf{T}}|X(t)|>\ve\right\}\leq N_1N_2 \left(\frac{e}{p}\right)^Q
\exp{\left\{-\frac{\ve^2(1-p)}{2\left(T_1^{2H_1}T_2^{2H_2}+\frac{p}{1-p} 4^{1-H}T_\eta^4\right)}
\right\}}, \ve>0.
\end{equation}
\end{prop}
\begin{proof}
Recall that $\rho_2(\mathbf{s},\mathbf{t})=|s_1-t_1|^{H_1}+|s_2-t_2|^{H_2},\mathbf{s}=(s_1,s_2),\mathbf{t}=(t_1,t_2),\mathbf{t},\mathbf{s}\in \mathbf{T}.$ From Lemma \ref{lemma_2} we get
$$\sup_{\rho(\mathbf{s},\mathbf{t})\le h}\left(\E\left(X(\mathbf{t})-X(\mathbf{s})\right)^2\right)^{1/2}\le \sup_{\rho(\mathbf{s},\mathbf{t})\le h}\left(t_2^{H_2}|s_1-t_1|^{H_1}+t_1^{H_1}|s_2-t_2|^{H_2}\right)\le T_\eta h.$$
Thus, we can put $\sigma(h)=T_\eta h$ in Theorem \ref{main_thrm}.  From (\ref{defin_beta_gama}) we have
$$\beta=\sigma\left(\left(\frac{T_1}{2}\right)^{H_1}+\left(\frac{T_2}{2}\right)^{H_2}\right)
=T_\eta\left(\left(\frac{T_1}{2}\right)^{H_1}+\left(\frac{T_2}{2}\right)^{H_2}\right).$$
It is clear that
$$\gamma^2=\sup_{\mathbf{t}\in \mathbf{T}}\E X^2(\mathbf{t})=T_1^{2H_1}T_2^{2H_2}\E X^2(\mathbf{1})=T_1^{2H_1}T_2^{2H_2}.$$
From Lemma \ref{prop_metric} we have
$$N(u)\leq 2 \left(\frac{T_1}{4 K_1 u^{\frac{1}{H_1}}}+\frac32\right)\left(\frac{T_2}{4 K_2 u^{\frac{1}{H_2}}}+\frac32\right),$$
and therefore
$$N(\sigma^{-1}(u))\leq 2 \left(\frac{T_1 {T_\eta}^{\frac{1}{H_1}}}{4 K_1 u^{\frac{1}{H_1}}}+\frac32\right)\left(\frac{T_2 {T_\eta}^{\frac{1}{H_2}}}{4 K_2 u^{\frac{1}{H_2}}}+\frac32\right).$$
It follows from $\beta>\beta p\ge u$ that
$$1<\left(\left(\frac{T_1}{2}\right)^{H_1}+\left(\frac{T_2}{2}\right)^{H_2}\right)^{\frac{1}{H_i}} \frac{T_\eta^{\frac{1}{H_i}}}{u^{\frac{1}{H_i}}}, \quad i=1,2.$$
Recall that $0<H_i<1$ and
$$\frac{T_i}{2}=\left(\left(\frac{T_i}{2}\right)^{H_i}\right)^{\frac{1}{H_i}}\leq\left(\left(\frac{T_1}{2}\right)^{H_1}+\left(\frac{T_2}{2}\right)^{H_2}\right)^{\frac{1}{H_i}}, \quad i=1,2.$$
Then
$$\left(\frac{T_i {T_\eta}^{\frac{1}{H_i}}}{4 K_i u^{\frac{1}{H_i}}}+\frac32\right)\leq  \frac{T_i {T_\eta}^{\frac{1}{H_i}}}{4 K_i u^{\frac{1}{H_i}}}+\left(\left(\frac{T_i}{2}\right)^{H_i}+\left(\frac{T_2}{2}\right)^{H_2}\right)^{\frac{1}{H_1}} \frac{3(T_\eta)^{\frac{1}{H_i}}}{2u^{\frac{1}{H_i}}}$$
$$\leq \left(\left(\frac{T_1}{2}\right)^{H_1}+\left(\frac{T_2}{2}\right)^{H_2}\right)^{\frac{1}{H_i}} \frac{T_\eta^{\frac{1}{H_i}}}{u^{\frac{1}{H_i}}}\left( \frac{1}{2 K_i }+\frac{3}{2}\right).$$
Therefore, we have the following inequality for $Z(p),$ where $Z(p)$ is defined in (\ref{zp}). For each \\$0<\mu<1/Q$ we obtain
$$Z(p)\le\left(\frac{1}{\beta p} \int_0^{\beta p}\left(\left(\left(\frac{T_1}{2}\right)^{H_1}+\left(\frac{T_2}{2}\right)^{H_2}\right)^Q \frac{T_\eta^Q}{u^Q}\frac{N_1 N_2}{2}\right)^{\mu}du\right)^{1/\mu}$$
$$=2 N_1 N_2\left(\left(\frac{T_1}{2}\right)^{H_1}+\left(\frac{T_2}{2}\right)^{H_2}\right)^Q\frac{ T_\eta^Q }{(\beta p)^{1/\mu}}\left(\int_0^{\beta
p}\frac{1}{u^{Q\mu}}\right)^{1/\mu}$$
$$=
\frac{N_1 N_2}{2}\left(\left(\frac{T_1}{2}\right)^{H_1}+\left(\frac{T_2}{2}\right)^{H_2}\right)^Q\frac{T_\eta^Q }{(\beta p)^{Q}}\left(\frac{1}{1-Q \mu}\right)^{1/\mu}.$$  As
$\mu\rightarrow0$, we have
$$Z(p)\le \frac{N_1 N_2}{2}\left(\left(\frac{T_1}{2}\right)^{H_1}+\left(\frac{T_2}{2}\right)^{H_2}\right)^Q\frac{ T_\eta^Q }{(\beta p)^{Q}}e^Q=\frac{N_1 N_2}{2}\left(\frac{e}{p}\right)^Q.$$ Finally, from (\ref{cor_inequal_1}) we obtain
$$I_{\mathbf{T}}(\ve)\leq N_1N_2 \left(\frac{e}{p}\right)^Q
\exp{\left\{-\frac{\ve^2(1-p)}{2\left(T_1^{2H_1}T_2^{2H_2}+\frac{p}{1-p}T_\eta^2\left(\left(\frac{T_1}{2}\right)^{H_1}+\left(\frac{T_2}{2}\right)^{H_2}\right)^2\right)}
\right\}}$$
$$\leq N_1N_2 \left(\frac{e}{p}\right)^Q
\exp{\left\{-\frac{\ve^2(1-p)}{2\left(T_1^{2H_1}T_2^{2H_2}+\frac{p}{1-p} 4^{1-H}T_\eta^4\right)}
\right\}}.$$
\end{proof}
\begin{cor}
\label{cor_pr2_2} Under the conditions of Proposition \ref{prop_2} we have
\begin{equation}
\label{cor_pr2_2_ineq}
\begin{gathered}
\pr\left\{\sup_{\mathbf{t}\in\mathbf{T}}|X(t)|>\ve\right\}\leq N_1N_2\ve^{2 Q}
\exp\left\{Q+\frac{3}{2T_1^{2H_1}T_2^{2H_2}\left(3+ 4^{1-H}\right)}\right\}\times\\
\exp{\left\{-\frac{3\ve^2}{2\left(3T_1^{2H_1}T_2^{2H_2}+ 4^{1-H}T_\eta^4\right)}
\right\}},\quad  \ve>2.
\end{gathered}
\end{equation}
\end{cor}
\begin{proof}
The corollary follows from (\ref{prop_2_ineq}) if we put
$p=1/{\ve^2}.$
\end{proof}
Consider probability distribution of extremes defined in $[0,1]^2$.
\begin{cor}
\label{cor_pr2_3} Let $(\mathbf{T},\rho)=([0,1]^2,\rho_2).$ Under the conditions of Proposition \ref{prop_2} we have
\begin{equation*}
\label{cor_pr2_3_ineq}
\pr\left\{\sup_{\mathbf{t}\in[0,1]^2}|X(t)|>\ve\right\}\leq N_1N_2\ve^{2 Q}
\exp\left\{Q+\frac{3}{2\left(3+ 4^{1-H}\right)}\right\}\exp{\left\{-\frac{3\ve^2}{2\left(3+ 4^{1-H}\right)}
\right\}}, \quad \ve>2.
\end{equation*}
\end{cor}
\begin{proof}
In this case $T_1=T_2=1,$ so the corollary follows from (\ref{cor_pr2_2_ineq}).
\end{proof}
We want to find an upper bound for probability distribution of extremes defined in $[1,+\infty)^2$. For this goal we obtain probabilities defined in $[1,2]^2.$
\begin{prop}
\label{prop_3} Let $\mathbf{T}=[1,2]^2,~ \rho=\rho_2$ and
$X=\{X(\mathbf{t}),\mathbf{t}\in \R^2_+\}$ be a centered self-similar Gaussian
random field with stationary rectangular increments. Under the conditions of
Theorem \ref{main_thrm} for all
$0<p<1$ we have
\begin{equation}
\label{prop_3_ineq} I_{[1,2]^2}(\ve)=\pr\left\{\sup_{\mathbf{t} \in
[1,2]^2}{|X(t)|>\ve}\right\}\le N_1 N_2\left(\frac{e}{p}\right)^Q
\exp{\left\{-\frac{\ve^2(1-p)}{2\left(4^{H_1+H_2}+\left(1+2^{|H_1-H_2|}\right)^2\frac{p}{1-p}\right)}\right\}}.
\end{equation}
%where $\eta=\max\{H_1,H_2\},~H=\min\{H_1,H_2\}.$
\end{prop}
\begin{proof}
We prove the proposition in the same way as Proposition \ref{prop_2}. Denote $\eta=\max\{H_1,H_2\}$ and $H=\min\{H_1,H_2\}.$
It is clear that $\sigma(h)=2^\eta h$ and
$$\beta=\sigma\left(\left(\frac{1}{2}\right)^{H_1}+\left(\frac{1}{2}\right)^{H_2}\right)
=2^\eta\left(2^{-H_1}+2^{-H_2}\right)=1+2^{|H_1-H_2|},$$
 $$\gamma^2=4^{H_1+H_2}.$$
From Lemma \ref{prop_metric} we have
$$N(\sigma^{-1}(u))\leq 2 \left(\frac{2^{\eta/H_1}}{4 K_1 u^{1/H_1}}+\frac32\right)\left(\frac{2^{\eta/H_2}}{4 K_2 u^{1/{H_2}}}+\frac32\right).$$

It follows from $\beta>\beta p\ge u>0$ that
$$1\leq \frac{\beta^{1/H_1}}{u^{1/H_1}}=\frac{\left(1+2^{|H_1-H_2|}\right)^{1/H_1}}{2u^{1/H_1}}.$$
Then for $i=1,2$
$$\left(\frac{2^{\eta/H_i}}{4 K_i u^{1/{H_1}}}+\frac32\right)\leq  \frac{2^{\eta/H_i}}{4 K_i u^{1/{H_i}}}+ \frac{3\left(1+2^{|H_1-H_2|}\right)^{1/H_i}}{2u^{1/{H_i}}}\leq  \frac{\left(1+2^{|H_1-H_2|}\right)^{1/{H_i}}}{u^{1/{H_i}}}\left( \frac{1}{2 K_i }+\frac{3}{2}\right).$$
Further, from definition (\ref{zp}) of Z(p) we get the following inequality.
$$Z(p)\le\left(\frac{1}{\beta p} \int_0^{\beta p}\left(\left(1+2^{|H_1-H_2|}\right)^Q \frac{N_1 N_2}{2u^Q}\right)^{\mu}du\right)^{1/\mu}$$
$$=\frac{N_1 N_2}{2}\left(1+2^{|H_1-H_2|}\right)^Q\frac{1}{(\beta p)^{1/\mu}}\left(\int_0^{\beta
p}\frac{1}{u^{Q\mu}}\right)^{1/\mu}=
\frac{N_1 N_2}{2 p^{Q} }\left(\frac{1}{1-Q \mu}\right)^{1/\mu}.$$
As $\mu\rightarrow0$, we have
$$Z(p)\le \frac{N_1 N_2}{2}\left(\frac{e}{p}\right)^Q.$$
Thus, we obtain
$$I_{[1,2]^2}(\ve)\le N_1 N_2\left(\frac{e}{p}\right)^Q
\exp{\left\{-\frac{\ve^2(1-p)}{2\left(4^{H_1+H_2}+\left(1+2^{|H_1-H_2|}\right)^2\frac{p}{1-p}\right)}\right\}}.$$
\end{proof}
As before, denote $\eta=\max\{H_1,H_2\}.$
\begin{cor}
\label{cor_pr3_2} Under the conditions of Proposition \ref{prop_3} for $\ve>2$ we have
\begin{equation}
\label{cor_pr3_2_ineq}
I_{[1,2]^2}(\ve)\le N_1 N_2\exp\{Q+\frac{1}{2\left(4^{H_1+H_2}+1\right)}\}\ve^{2Q}
\exp{\left\{-\frac{3\ve^2}{2 \cdot 4^\eta\left(4^{H}3+4^{1-H}\right)}\right\}}.
\end{equation}
\end{cor}
\begin{proof}
The corollary follows from (\ref{prop_3_ineq}), if we put
$p=1/{\ve^2}.$
\end{proof}
\begin{theor}
\label{theor_m}
Let $\mathbf{T}=[1,\infty)^2,~ \rho=\rho_2$ and
$X=\{X(\mathbf{t}),\mathbf{t}=(t_1,t_2)\in \R^2_+\}$ be a centered self-similar Gaussian
random field with stationary rectangular increments. Let
$\varphi: (0,+\infty)^2\rightarrow (0,+\infty)$ be an increasing
function in each coordinate. Suppose that for any $D>0$
\begin{equation}
\label{theor_series_1} \sum_{n=0}^\infty\sum_{m=0}^\infty
{\exp{\left\{-D\varphi\left(2^{n},2^{m}\right)\right\}}}<+\infty.
\end{equation}
Denote $$C_1=N_1 N_2\exp\{{Q+\frac{1}{2(4^{H_1+H_2}+1)}}\}~ and~
C_2=\frac{3}{2\cdot4^\eta\left(4^{H}3+4^{1-H}\right)}.$$
If $\ve>\frac{2}{\varphi(\mathbf{1})}$, then
\begin{equation}
\label{theor m_ineq}
Y(\ve):=\pr\left\{\sup_{\mathbf{t}\in[1,+\infty)^2}\frac{|X(\mathbf{t})|}{t_1^{H_1}t_2^{H_2}\varphi(\mathbf{t})}>\ve\right\}
\le
C_1\ve^{2Q}\sum_{n=0}^\infty\sum_{m=0}^\infty\frac{\varphi^{2Q}\left(2^{n},2^{m}\right)}
{\exp{\left\{C_2\ve^2\varphi^2\left(2^{n},2^{m}\right)\right\}}}.
\end{equation}
\end{theor}
\begin{proof}
At first, we have the following obvious inequality
$$\pr\left\{\sup_{\mathbf{t}\in[1,+\infty)^2}\frac{|X(\mathbf{t})|}{t_1^{H_1}t_2^{H_2}\varphi(\mathbf{t})}>\ve\right\}\le\sum_{n=1}^\infty\sum_{m=1}^\infty\mathop{\pr\left\{\sup_{ \substack{t_1\in[2^{n-1},2^n]  \\
t_2\in[2^{m-1},2^m]
}}\frac{|X(\mathbf{t})|}{t_1^{H_1}t_2^{H_2}\varphi(\mathbf{t})}>\ve\right\}}.$$
Then from monotonicity of $\varphi$ we get for all $n,m>1:$
$$\mathop{\pr\left\{\sup_{ \substack{t_1\in[2^{n-1},2^n]  \\
t_2\in[2^{m-1},2^m]
}}\frac{|X(\mathbf{t})|}{t_1^{H_1}t_2^{H_2}\varphi(\mathbf{t})}>\ve\right\}}\leq \mathop{\pr\left\{\sup_{ \substack{t_1\in[2^{n-1},2^n]  \\
t_2\in[2^{m-1},2^m]
}}\frac{2^{(1-n)H_1}2^{(1-m)H_2}|X(\mathbf{t})|}{\varphi(2^{n-1},2^{m-1})}>\ve\right\}}.$$
By self-similarity, we obtain the following equality for all $n,m\geq 1:$
$$ \mathop{\pr\left\{\sup_{ \substack{t_1\in[2^{n-1},2^n]  \\
t_2\in[2^{m-1},2^m]
}}\frac{2^{(1-n)H_1}2^{(1-m)H_2}|X(\mathbf{t})|}{\varphi(2^{n-1},2^{m-1})}>\ve\right\}}=\mathop{\pr\left\{\sup_{\mathbf{t}\in [1,2]^2}\frac{|X(\mathbf{t})|}{\varphi(2^{n-1},2^{m-1})}>\ve\right\}}.$$
Thus,
$$Y(\ve)\leq \sum_{n=1}^\infty\sum_{m=1}^\infty\mathop{\pr\left\{\sup_{\mathbf{t}\in [1,2]^2}\frac{|X(\mathbf{t})|}{\varphi(2^{n-1},2^{m-1})}>\ve\right\}}$$
$$=\sum_{n=1}^\infty\sum_{m=1}^\infty\pr\left\{\sup_{\mathbf{t}\in[1,2]^2}
|X(\mathbf{t})|>\ve\varphi\left(2^{n-1},2^{m-1}\right)\right\}.$$
It follows from Corollary \ref{cor_pr3_2} that for $\ve>\frac{2}{\varphi(\mathbf{1})}$ we have
$$Y(\ve)\le
C_1\ve^{2Q}\sum_{n=0}^\infty\sum_{m=0}^\infty\varphi^{2Q}\left(2^{n},2^{m}\right)
\exp{\left\{-C_2\ve^2\varphi^2\left(2^{n},2^{m}\right)\right\}}.$$
This completes the proof.
\end{proof}
\begin{cor}
\label{main_cor_2}
%Let conditions of Theorem \ref{theorem_3} hold.
If for any $\mathbf{H}\in (0,1)^2$ the series $$\sum_{n=0}^\infty\sum_{m=0}^\infty\frac{\varphi^{2Q}\left(2^{n},2^{m}\right)}{
\exp{\left\{2\frac{\varphi^2\left(2^{n},2^{m}\right)}{\varphi^2\left(\mathbf{1}\right)}\right\}}}<+\infty,$$ then for $\ve>\frac{2}{\varphi(\mathbf{1})}\sqrt{\frac{2}{4^\eta 3}\left(4^H 3+4^{1-H}\right)},$
\begin{equation}
\label{main_cor_2_ineq}
\begin{gathered}
Y(\ve)\leq C_1 \ve^{2Q}\exp\left\{-\frac{\ve^2}{2}\frac{\varphi^2(\mathbf{1})}{C_2}\right\} \sum_{n=0}^\infty\sum_{m=0}^\infty\frac{\varphi^{2Q}\left(2^{n},2^{m}\right)}{
\exp{\left\{2\frac{\varphi^2\left(2^{n},2^{m}\right)}{\varphi^2\left(\mathbf{1}\right)}\right\}}}.
\end{gathered}
\end{equation}
\end{cor}
\begin{proof}
Denote
$$u=\frac{3}{4\cdot4^\eta\left(4^{H}3+4^{1-H}\right)}\varphi^2(\mathbf{1})\quad  \text{ and } \quad v_{n,m}=2\frac{\varphi^2(2^n,2^m)}{\varphi^2(\mathbf{1})},n,m\geq 0.$$

It can easily be checked that $u\ve^2>2$ and $v_{n,m}>2,n,m\geq 0.$
Recall that for $u\ve^2,v_{n,m}>2$ we have $u\ve^2+v_{n,m}\leq u\ve^2 v_{n,m}.$
It follows from (\ref{theor m_ineq}) that for $\ve>\frac{2}{\varphi(\mathbf{1})}\sqrt{\frac{2}{4^\eta 3}\left(4^H 3+4^{1-H}\right)}>\frac{2}{\varphi(\mathbf{1})}$ we have
$$
Y(\ve)\leq C_1 \ve^{2Q}\exp\left\{\frac{3\ve^2}{4\cdot4^\eta\left(4^{H}3+4^{1-H}\right)}\varphi^2(\mathbf{1})\right\} \sum_{n=0}^\infty\sum_{m=0}^\infty\frac{\varphi^{2Q}\left(2^{n},2^{m}\right)}{
\exp{\left\{2\frac{\varphi^2\left(2^{n},2^{m}\right)}{\varphi^2\left(\mathbf{1}\right)}\right\}}}.
$$
The corollary is proved.
\end{proof}
We present the example of applying Corollary \ref{main_cor_2}.

\textbf{ Example 2.} Let  $\varphi_1,\varphi_2$  be the positive functions of  $R^2_+$ to $R$ such that
$$\varphi_1(\mathbf{x})=\sqrt{(2+\delta)}\sqrt{\ln(\log_2{(x_1x_2)}+e)}, \mathbf{x}=(x_1,x_2)\in \R^2_+$$
and
$$\varphi_2(\mathbf{x})=\sqrt{(2+\delta)}\sqrt{\ln (e+\log_2x_1)+\ln (e+\log_2x_1)}, \mathbf{x}=(x_1,x_2)\in \R^2_+.$$
Then
$$\varphi_1(2^n,2^m)=\sqrt{\ln(n+m+e)}, n,m\in \{0\}\cup \mathbb{N},$$
$$\varphi_2(2^n,2^m)=\sqrt{\ln (n+e)+\ln(m+e)},n,m\in \{0\}\cup \mathbb{N},$$
and $\varphi_1(\mathbf{1})=\varphi_2(\mathbf{1})=1.$

Therefore, from (\ref{main_cor_2_ineq}) we get
$$\sum_{n=0}^\infty\sum_{m=0}^\infty\frac{\varphi_1^{2Q}\left(2^{n},2^{m}\right)}{
\exp{\left\{2\frac{\varphi_1^2\left(2^{n},2^{m}\right)}{\varphi_1^2\left(\mathbf{1}\right)}\right\}}}=\sum_{n=0}^\infty\sum_{m=0}^\infty\frac{\left(\ln(n+m+e)\right)^{Q}}{
(n+m+e)^2},$$
$$\sum_{n=0}^\infty\sum_{m=0}^\infty\frac{\varphi_2^{2Q}\left(2^{n},2^{m}\right)}{
\exp{\left\{2\frac{\varphi_2^2\left(2^{n},2^{m}\right)}{\varphi_2^2\left(\mathbf{1}\right)}\right\}}}=\sum_{n=0}^\infty\sum_{m=0}^\infty\frac{\left(\ln(n+e)(m+e)\right)^{Q}}{(n+e)^2(m+e)^2}.$$

Hence, from Corollary \ref{main_cor_2} we have
$$\pr\left\{\sup_{\mathbf{t}\in[1,+\infty)^2}\frac{|X(\mathbf{t})|}{t_1^{H_1}t_2^{H_2}\varphi_1(\mathbf{t})}>\ve\right\}
\le C_1 \ve^{2Q}\exp\left\{-\frac{C_2}{2}\ve^2\right\} \sum_{n=0}^\infty\sum_{m=0}^\infty\frac{\left(\ln(n+m+e)\right)^{Q}}{
(n+m+e)^2},$$
$$\pr\left\{\sup_{\mathbf{t}\in[1,+\infty)^2}\frac{|X(\mathbf{t})|}{t_1^{H_1}t_2^{H_2}\varphi_2(\mathbf{t})}>\ve\right\}
\le C_1 \ve^{2Q}\exp\left\{-\frac{C_2}{2}\ve^2\right\}  \sum_{n=0}^\infty\sum_{m=0}^\infty\frac{\left(\ln(n+e)(m+e)\right)^{Q}}{(n+e)^2(m+e)^2}.$$

Thus, we obtain probability distributions for extremes of a normalized self-similar Gaussian random field with stationary rectangular increments that defined in $[1,+\infty)^2$.

\end{document}